\documentclass[11pt]{article}
\usepackage[utf8]{inputenc}
\usepackage[T1]{fontenc}
\usepackage{lmodern}
\usepackage{bbm}
\usepackage{xcolor}
\usepackage{enumerate}
\usepackage{amsmath,amssymb,amsfonts,mathtools}
\usepackage{amsthm}
\newtheoremstyle{mydefstyle} 
{\topsep}                    
{\topsep}                    
{\upshape}                   
{}                           
{\bfseries}                   
{.}                          
{.5em}                       
{}
\newtheorem{theorem}{Theorem}

\newtheorem{lemma}{Lemma}
\theoremstyle{mydefstyle}

\newcommand{\R}{\ensuremath{\mathbb{R}}}

\usepackage[hang,bottom]{footmisc}

\usepackage{hyperref}
\usepackage{tikz}
\begin{document}

\title{The Unified Approach for Best Choice Modeling Applied to
  Alternative-Choice Selection Problems}

\author{Rémi
  Dendievel\footnote{Postal address: Département de Mathématique,
    Faculté des Sciences, Université Libre de Bruxelles, Campus de la
    Plaine - CP 210, Boulevard du Triomphe, 1050 Bruxelles,
    Belgium. Email address: redendie@ulb.ac.be}} \date{} 

\maketitle
\begin{abstract}
  \noindent
  The objective of this paper is to show that the so-called
  \emph{unified approach} to stopping problems with unknown cardinality
  introduced in Bruss (1984) proves to be efficient for solving other
  types of best-choice problems. We show that what we will call the
  \emph{alternative-choice} stopping problem, which will be exemplified right away in
  Section~\ref{sec:theor-backgr}, can be seen as a ``two-sided''
  Secretary problem. This problem is instigated by a former problem of
  R.~R.~Weber (Cambridge University). Our approach yields for unknown cardinality
  the sharp lower bound $1/2$ for the probability of success. This problem is, at
  the same time, a special case of a model more generally based on $k\ge 2$
  linearly ordered subsets. We shall also give the solution for such problems
  for $k$ independent streams of arrivals.   Our approach is elementary and
  self-contained.

  

\end{abstract}

\section{Motivating Example and Theoretical Background}
\label{sec:theor-backgr}

Consider an interviewer who interviews sequentially streams of 
candidates of men and women both arriving in $[0,1]$. The total numbers
of men and women are random variables and their arrival times are also
i.i.d. The interviewer can rank in a unique way those of the same gender.
At any time that a candidate presents itself the interviewer may stop and
choose that candidate. Similarly to the well-known Secretary Problem, he
is successful if he selects either the best of the men, or the best of
the women. This is an example, with $k=2$, of what we call an \emph{alternative
best-choice stopping problem} (which was motivated by Weber, 2012).
We show that the unified approach to stopping problems~(Bruss, 1984)
can be applied to this problem. In the case of the men and women, and
numbers of men and women are a priori unknown, we find that by choosing
the first candidate to arrive after time $1/2$ who is best of its gender
so far, the probability of success has the sharp lower bound of $1/2$.

\bigskip
\noindent\textbf{Unified approach.} An unknown number $N>0$ of
uniquely rankable options arrive on $[0,1]$  with independent and identically
distributed arrival times according to a continuous distribution function $F$
defined on $[0,T]$.
One speaks of a \emph{success} if the decision maker accepts the absolute rank~1,
or in other words, the very last relatively best (record).
Let $\sigma_t$ be the strategy to wait until time $t$ and then to select the first
record (relative best) option thereafter, if any (the stopping time induced by $\sigma_t$ is set
equal to $T,$ otherwise). We first recall the $1/e$-law of Bruss~(1984).
Let 
\begin{equation}
  \label{1overElaw}
  t^\star = \inf\{ t:  F(t) = 1/e \}=:F^{-1}(1/e).
\end{equation}
Then, the $1/e$-law says: First, the strategy to stop on the first record appearing
at time $t^\star$ or later, $\sigma_{t^\star}$, say, succeeds with probability at least
$1/e$. Second, conditioned on $\{N=n\}$ the bound $1/e$ is the limit
of decreasing success probabilities as $n \to \infty.$ Third, the
probability of $\sigma_{t^\star}$ (called $1/e$-strategy) selecting no
candidate at all is exactly $1/e$. It is for this {\em triple} coincidence of
the number $1/e$ that the author called the result ``$1/e$-law''. This result
came as a surprise since, in general, a success probability near $1/e$ seemed at
that time out of reach. See the comments of Samuels (1985) in \emph{Mathematical Reviews}.

\subsubsection*{Alternative choice and related work}
\label{sec:related-work}
What we will call the \emph{problem of alternative choice} is
the  continuous-time model of Bruss (1984) but now with $k\ge 1$ classes of
options.
For example the problem of choosing with one choice either the best male, or the
best female candidate. Observed options are supposed to be uniquely rankable among
all options in the same class but typically not comparable with options from
another class.

This version of Weber's problem as well as the underlying work may be seen as a
selection model based on partially ordered sets (posets), or, as we will argue,
also as a multicriteria selection model. There are several interesting
articles which are related with our work. Selection problems involving
posets seem to have been first studied by Stadje~(1980). Morayne et
al.~(2008) propose a novel interesting universal algorithm for
posets where the cardinality of the set is 
known. The ordering structure of the poset need not be known; in
particular, the number of maximal elements in the set need not be known.

Freij and Wästlund~(2010) present a generalized strategy in
continuous time for a similar model to theirs, where the size of the
poset is not fixed. This is partially in the spirit of the underlying
paper following the unified approach. Their work can be seen as a
generalization of the previous work.  Interestingly, they also obtain the
ubiquitous $1/e$ lower bound for the success probability in quite a
general setting.

Kumar et al.~(2011) and Garrod and Morris~(2013) studied independently
a similar stopping problem on a poset of $n$ elements where both $n$
and the number $k$ of maximal elements are known. In that case, the
probability of success of the strategy they propose can be improved to
$k^{-1/(k-1)}$, a value which is bound to turn up in our generalization by
the unified approach.

Gnedin's multicriteria models~(Gnedin, 1981) share some similarity with our
model. As he explained: it is unlikely that a secretary is best in
several criteria at the same time; the probability that this happens
tends to $0$ as the number $n$ of applicants tends to infinity. That
is why the model is similar to one where each secretary belongs to
only one category.

In selection problems with more information on the number of arrivals
of options, so-called quasi-stationarity is a very desirable property
for obtaining closed-form solutions. The approach of Bruss and
Samuels~(1990) can be adapted for options stemming from different
classes, but here again the unified approach without any knowledge at all is
not that much weaker in the result and therefore, taking things
together, clearly preferable for applications.

Parts of the precedingly cited articles go deeper into certain
directions not mentioned here. In particular our result is partially
intrinsic in the interesting paper of Garrod and Morris (2013). However,
first, the present article shows that the unified approach on which we
focus yields the results in an elementary and self-contained way, and
has, in particular for the case $k= 2$ much appeal for applications in
real life. This is seen in several quite different examples in Bruss and
Ferguson~(2002) for the case $k=1$. Second, according to their list of
references, none of the authors on poset selection problems cited above
has referenced the unified approach of 1984 which adds to our motivation
for this article.

\section{Main results}
\label{sec:main-result}

{\em Threshold-time strategy.} Similarly to what has been done in the
$1/e$-law, we aim to find a time threshold $t^\star$ after which we
decide to stop on the first encountered record, regardless of the
class to which this record belongs. Let $\sigma_t$ be the strategy to
wait until time $t$ and then to select the first record option
thereafter, if any; the stopping time induced by $\sigma_t$ is set
equal to $T$ otherwise, meaning a complete failure (no choice at all).

\medskip
\noindent\emph{Model.\ }Let $N_1,N_2,\dotsc,N_k\in \{1, 2, \dots \}$
with no other information available on these numbers than being positive integers.
They are seen as the numbers of options of class~$1$, respectively class~$2$,\dots,
and so on. We consider iid arrival times $X_{l,i}$ ($1\le i\le N_k$, $1\le
l\le k$) of the options on $[0,1]$ so that the $r$th option of type
$j$ arrives at time $X_{j,r}$. Since there are $k$ non-empty classes, there are
also $k$ options which are best in their class by definition, because we supposed
that they can be ranked in a unique way. These will be called ``maxima''.
We would like to stop online on any of the maxima.

\subsection*{The $1/2$-rule}
\label{sec:particular-case-}

In the view of real life applications, the case $k=2$ is particularly
appealing. First, the decision rule will be easy to
remember. Second, it also covers the problem predicting the best or alternatively
the worst outcome from  a stream of events. Indeed, we can construct the two
streams artificially as follows. Consider one stream of rankable secretaries and
remember the very first secretary. Then, all the subsequent secretaries with a better
rank than the remembered secretary belong to ``class~1'' and the secretaries with
a worse rank will belong to class~2. Hence this serves the problem of finding
either the best or the worst of all. We also mention here that Bayon et al.~(2016)
have studied different cases of selecting the best or worst secretary from a
sequence of secretaries with random length.

If one single stream of options (not necessarily secretaries) can be divided, on
some rational, in two complementary classes (``positive'' and ``negative''), then
this construction applies of course as well. In the case of a financial index,
positive increments belong in one class, negative increments belong to the second
class.

Theorem~\ref{thm:1-over-2-law} is Theorem~\ref{thm:k-class} with $k=2$, but we state
Theorem~\ref{thm:1-over-2-law} first because of its appealing simplicity.

\begin{theorem}
  \label{thm:1-over-2-law}
  In the model described above in Section~\ref{sec:main-result}, if
  there are two non-empty classes of options, then stopping time
  $\sigma_{1/2}$ which stops on the first record belonging to either of the
  two classes after time $1/2$ has success probability of at
  least $1/2$. (This holds thus for any number of options as long as both
  classes are non-empty.)
\end{theorem}

 \medskip \noindent The proof of Theorem~\ref{thm:1-over-2-law} is immediate from the proof
 following Theorem~\ref{thm:k-class} 
and can therefore be postponed. 
 
 \medskip
 \noindent{\bf Remark 1.\ } Here again, this is the best fixed-threshold
 	strategy yielding the success probability $1/2$ for such strategies. It
	cannot be improved because the probability of success of $\sigma_{1/2}$ converges
  to $1/2$ if the numbers of options in both classes tend to infinity.
  This confirms the result in the discrete setting from Garrod and Morris (2013).

\subsection*{The general case}
\label{sec:general-case}

\begin{theorem}\label{thm:k-class}
	Let $N_1,N_2,\dots,N_k\ge 1$ denote the unknown number of options
	of class~$1$, respectively class~$2$, \dots, respectively class~$k$.
	We assume that the arrival times of all options are independent and that
	their distribution is uniform on $[0,1]$. We suppose that inside each class,
	the options are uniquely rankable against one another.
	
  	Then, the strategy $\sigma_{t_k}$ which stops on the first record
  	(=relative maximum in its class) appearing after time $t_k
  	:= k^{-1/(k-1)}$ satisfies
  \begin{enumerate}[(a)]
  \item $\sigma_{t_k}$ succeeds with probability at least $t_k.$
  \item Seen as a function of the unknowns $N_1=n_1$, $N_2=n_2$,
    \dots, $N_k=n_k$, the success probability decreases monotonically
    in all $n_j$ with limit $t_k$ as $\min_{1\le j\le k} n_j \to
    \infty$.
  \end{enumerate}
\end{theorem}

\begin{proof}
  \textbf{(a)} Define $p(t)$ as the probability that the strategy
  $\sigma_t$ will turn out (at time $1$) to be a success.
  (Remark~2 is more specific about issues related to this
  ``probability''.)
  Denote by $T_1,T_2, \dots,T_k$ the arrival
  times of the maxima of class~$1$, respectively class~$2$, \dots,
  respectively class~$k$.  Consider the events
  \begin{equation}
    \label{eq:E_it}
   E_i(t):= \{T_1,\dots,T_{k-i}<t\}
   		\cap \{T_{k-i+1},\dots,T_{k-1},T_k\ge t\},
	\: \text{for $i=1,2,\dots,k$},
  \end{equation}
  and let
  \begin{equation}\label{eq:S_it}
  S_i(t) = E_i(t) \cap \{\sigma_t \text{ turns out to be a success}\}.
  \end{equation}
  We have $p(t) = \sum_{i=1}^k \binom{k}{i} P(S_i(t))$,
  because there are $\binom{k}{k-i}$ ways to choose $k-i$
  i.i.d.\ random arrival times out of a set of $k$ arrival times.
  Since the classes are non-empty, this is independent of their size since
  a maximum in each class must exist.
 
  Let $T_{(i)} := \min(T_{k-i+1}, \dots, T_k)$ be the arrival time
  of the first maxima, restricted to the last $i$ classes, arriving in $[t,1]$.
  We notice that when $E_i(t)$ happens, the strategy leads to a success if
  stopping occurs at time $T_{(i)}$, by definition of $\sigma_t$.
  Because the $T_j$'s are uniformly distributed, the conditional 
  density of $T_{(i)}$ given that $E_i(t)$ happens, denoted by $f^{(i)}$, is
  \begin{equation}
    \label{eq:fi-density}
    f^{(i)}(s) = \chi_{[0,1]}(s) i(1-s)^{(i-1)} (1-t)^{-i},\quad s\in\R,
  \end{equation}
  where $\chi_A$ denotes the indicator function of a set $A$. We may
  now write
  \begin{align}
    p(t) &= \sum_{i=1}^{k} \binom{k}{i} P(\sigma_t\text{
      succeeds}|E_i(t)) P(E_i(t))
    \nonumber\\
    &= \sum_{i=1}^{k} \binom{k}{i} \biggl(\int_{t}^{1} f^{(i)}(s) P(\sigma_t\text{
      succeeds}|T_{(i)}=s,E_i(t)) \, ds \biggr)  t^{k-i}(1-t)^i \nonumber\\
    &= \sum_{i=1}^{k} \binom{k}{i} \biggl(\int_{t}^{1} f^{(i)}(s)
    P(A(t,s)|T_{(i)}=s,E_i(t)) \, ds \biggr) 
    t^{k-i}(1-t)^i, \label{eq:pt-decomp}
  \end{align}
  where $A(t,s)$ is defined as the event that no record of any class
  is seen on the interval $[t,s)$, for $0 \le t<s<1$. We now use
  \begin{equation}\label{eq:PAts-gen}
  P(A(t,s)|T_{(i)}=s,E_i(t))\ge (t/s)^i.
  \end{equation}
  To keep the current proof streamlined, we will prove~\eqref{eq:PAts-gen}
  in a technical lemma after the current proof.
	Using inequality~\eqref{eq:PAts-gen} in \eqref{eq:pt-decomp}, we obtain
  \begin{equation}\label{eq:pt-coeff-not-id}
    p(t) \ge \int_{t}^{1} \sum_{i=1}^{k} \binom{k}{i} t^{k-i}(1-t)^i
    i(1-s)^{(i-1)} (1-t)^{-i} (t/s)^i \, ds.
  \end{equation}
  Cancelling out a few terms, Equation~\eqref{eq:pt-coeff-not-id} simplifies
  to
  \begin{equation}\label{eq:pt-one-step-simplif}
  p(t) \ge \int_{t}^{1} t^{k}\sum_{i=1}^{k} \binom{k}{i} 
    i(1-s)^{(i-1)} s^{-i} \, ds.
  \end{equation}
  We can multiply and divide each term of the sum in~\eqref{eq:pt-one-step-simplif}
  by $(1-s)s^k$, so as to obtain
  \begin{equation}
	p(t) \ge \int_{t}^{1} \frac{t^{k}}{(1-s)s^{k}}\sum_{i=1}^{k}\binom{k}{i} 
    i(1-s)^{i} s^{k-i} \, ds,
  \end{equation}
  and the sum can be interpreted as the expectation of a random variable
  with distribution $\mathrm{Binomial}(k,1-s)$. Therefore the value of the
  sum $k(1-s)$.  Finally,
  \begin{equation}
    \label{eq:pt-bound-coeff-cki}
    p(t) \ge \int_{t}^{1} t^k k s^{-k} \, ds =
    \frac{k}{k-1}(t-t^k)=: h_k(t),
  \end{equation}
  The function $h_k(t)$ is maximized in $t=t_k$, hence $p(t_k)\ge
  h_k(t_k) = t_k$.

  \medskip
  \noindent\textbf{(b)} The proof of (a) is organized in such a way
  that we can argue here without further calculations.  Recall now
  \eqref{eq:pt-decomp} and . Note that conditioned
  on $N_j=n_j\ge n$, the event that of at least one arrival in class
  $j$ has an increasing probability as $n_j$ increases, for all
  $j$. But then the conditional probabilities of $A(t,s)$ in the
  integrands must be non-increasing and hence converge to the shown
  lower bounds $(t/s)^j$. Since this holds for all $0<t\le s<1$
  pointwise in $s,$ it must hold for the corresponding integrals. As
  they all have non-negative integrands, it must hold for their sum.
\end{proof}

\begin{lemma}
	Suppose that $n\ge 0$ random variables are independent and uniformly
	distributed on $[0,1]$. These random variables represent the arrival
	times of $n$ uniquely rankable options. The probability that no relative
	maximum is seen in the interval $[t,s)$ is greater or equal to $t/s$.
\end{lemma}
\begin{proof}
	We first prove the result for one class of options. 
	The following argument can already be found in Bruss and Yor~(2012)
	and in a slightly different form in Freij and Wästlund~(2010).
	Look at the following constellation of points on $[0,1]$.
	\begin{center}
	\bigskip
	\begin{tikzpicture}
		\draw[|-|] (0,0) node[below]{\scriptsize$0$}
				-- (5,0) node[below]{\scriptsize$1$}; 
		\fill (1,0) circle (1pt);
		\fill (1.3,0) circle (1pt);
		\fill (2.4,0) circle (1pt);
		\draw (2.2,0.05) -- ++(0,-0.1); 
		\draw (2.8,0.05) -- ++(0,-0.1); 
		\node at (2.2,-.2) {$t$}; 
		\node at (2.8,-.2) {$s$}; 
	\end{tikzpicture}
	\end{center}
	There are two possibilities. Either there is no arrival at all in $[0,s]$,
	in which case $A(t,s)$ happens automatically, or there is at least one. But
	if there is at least one, then there is a best among them. Since a conditional
	uniform random variable stays uniform on the given interval, any point has
	probability $t/s$ of arriving in $[0,t]$, given it arrives in $[0,s]$.
	Hence, since $t,s\le 1$, we have that the probability of no record (in 
	class 1) in $[t,s]$ is either $t/s$ or one, that is at least $t/s$.
	
	Now, more generally, in the context of Theorem~\ref{thm:k-class},
	this argument can be applied to any class of options.
	If we consider $i$ classes, by independence of the arrival 
	times, the probability of the event $A(t,s)$ that no relative maximum (record)
	from these $i$ classes is seen in $[t,s)$ is then simply at least $(t/s)^i$.
\end{proof}

\noindent\textbf{Remark 2.\ } Remember the definition of $p(t)$ at the beginning of
the proof of Theorem~\ref{thm:k-class}. Note that the event of the last record at
time $t$ in its class is not measurable in the filtration generated by the observations
up to time $t$. This probability would only be defined conditionally. The problem of
giving a lower bound is however not affected by this.
\medskip




\medskip
\noindent\textbf{Remark 3.\ }In the context of
Theorem~\ref{thm:k-class}, the probability of not stopping at all is
$k^{-1}t_k$. Hence for $k>1$ we cannot have another ``triple''
coincidence of the three crucial values as in the $1/e$-law for $k=1.$

\medskip

\noindent\emph{\bf Open problem.} Notice that
Theorem~\ref{thm:k-class} does not mention the word
optimality. Optimality is an open problem, as it also still is for the
$1/e$-law, as reconfirmed in Bruss and Yor (2012, Section~6.4). The
latter have shown optimality of their solution for the so-called
last-arrival problem which looks very similar, so that the conjecture
of optimality is now almost compelling. However, the complete lack of
information about the number of options is in record-based problems
even more delicate.

\section*{Acknowledgement}

The author is thankful to the referee for his very helpful comments which
improved the paper in many ways.




\end{document}